\documentclass[oneside, reqno,11pt]{amsart}

\usepackage{amsmath}
\usepackage{amssymb}
\allowdisplaybreaks[4]

\newtheorem{thm}{Theorem}[section]
\newtheorem{prop}[thm]{Proposition}

\newtheorem{cor}[thm]{Corollary}

\theoremstyle{definition}
\newtheorem{definition}[thm]{Definition}
\newtheorem{example}[thm]{Example}

\theoremstyle{remark}

\numberwithin{equation}{section}

\begin{document}

\title{Certain actions of finitely generated groups on higher dimensional noncommutative tori}

\author{Zhuofeng He}
\address{Graduate School of Mathematical Sciences, the University of Tokyo,
3-8-1 Komaba, Meguro-ku, Tokyo 153-8914, Japan.}
\email{hzf@ms.u-tokyo.ac.jp}


\begin{abstract}
We study certain actions of finitely generated abelian groups on higher dimensional noncommutative tori. Given a dimension $d$ and a finitely generated abelian group $G$, we apply a certain function to detect whether there is a simple noncommutative $d$-torus which admits a specific action by $G$. We also describe the condition of $G$ under which the associated crossed product is an AF algebra.

\end{abstract}

\maketitle\pagestyle{myheadings}\markboth{}{\hfill Certain actions of finitely generated groups on higher dimensional noncommutative tori\hfill}

\tableofcontents

\section{Introduction}
Certain finite group actions on the rotation algebras have already been systematically studied in \cite{Echterhorff} by S. Echterhoff, W. L\"uck, N. C. Phillips and S. Walters. It is known that up to conjugacy, a finite subgroup $F$ of ${\rm SL}_2(\mathbb{Z})$ satisfies $F=\mathbb{Z}_k$ where $k=2, 3, 4$ or $6$. According to their paper, by realizing certain generators for $\mathbb{Z}_k$, the action of $\mathbb{Z}_k$ on a rotation algebra can be given for each such $k$. We note that this action is the natural generalization of the action of ${\rm SL}_2(\mathbb{Z})$ on a torus $\mathbb{T}^2=\mathbb{R}^2/\mathbb{Z}^2$. If we require the rotation algebra to be a simple one, i.e. for the irrational rotation algebra, denoted by $\mathcal{A}_{\theta}$ where $\theta$ is an irrational number, the crosed product $\mathcal{A}_{\theta}\rtimes \mathbb{Z}_k$ is an AF-algebra for $k=2, 3, 4$ and $6$. Moreover for each $k$, the $K_0$-group $K_0(\mathcal{A}_{\theta}\rtimes \mathbb{Z}_k)$ has been calculated, which shows that $\mathcal{A}_{\theta}\rtimes \mathbb{Z}_k\cong \mathcal{A}_{\theta'}\rtimes \mathbb{Z}_{k'}$ if and only if $k=k'$ and $\theta=\theta'$ mod $\mathbb{Z}$, see \cite[Theorem 0.1]{Echterhorff}. Also they show that higher dimensional noncommutative tori admit flip actions of $\mathbb{Z}_2$, and for simple ones the crossed products are AF algebras.

As stated in their paper, the proof of the above theorem breaks into the following three steps. 
\begin{enumerate}
\item Computation of the K-theory of the crossed product.
\item Proof that the crossed product satisfies the Universal Coefficient Theorem.
\item Proof that the action has the tracial Rohlin property. 
\end{enumerate}
Their argument and proof in step (2) and (3) are general enough so that they are also effective for higher dimensional cases. Hence it is accessible if one obtains a proper realization of some finite group, a proper generalization of actions above to higher dimensional cases, and a way to compute the $K$-theory of the crossed product. In \cite{LEE}, the authors manage to obtain results in this way. They give a proper definition of such actions which generalize the action above, which they call it by \lq\lq canonical actions\rq\rq. Throughout this paper we are interested in actions of this type. A detailed discussion of such actions is in the Preliminaries. By realizing a cyclic group with the companion matrix of a cyclotomic polynomial, the authors study such actions of cyclic group on simple higher dimensional tori for certain dimensions. To finish step (1), i.e. to compute the $K$-groups of the associated crossed product, the authors first apply \cite[Lemma 2.1]{Echterhorff} and \cite[Theorem 4.1]{Packer} to write the crossed product in terms of another twisted group $C^*$-algebra, then to compute its $K$-theory, according to \cite[Theorem 0.3]{Echterhorff}, it is the same to compute the $K$-theory of untwisted group $C^*$-algebra, and finally by \cite[Theorem 0.1]{Luck} they manage to obtain desired the $K$-groups. The authors then show that for certain dimension $d$ and order $n$ there is a simple noncommutative $d$-torus on which $\mathbb{Z}_n$ acts and the $K_1$-group of the crossed product is not trivial, see \cite[Theorem 3.6]{LEE}. 

We continue their work and study such actions of finitely generated abelian groups on simple higher dimensional noncommutative tori. By combining the tensor structure of higher dimensional noncommutatvie tori and the realization method in \cite{LEE} for cyclic groups and actions, we manage to realize all of the possible actions under cyclic groups. Then for a given dimension $d$ by the structure of finite cyclic group of ${\rm GL}_d(\mathbb{Z}_d)$, we apply a certain function to tell when we can realize the action of the cyclic group on a simple noncommutative $d$-torus. Moreover we describe the condition under which the crossed product is an AF algebra. The function $W(n)$ is defined later. Hereby we remark that through this paper we consider noncommutative $d$-tori, where we usually suppose $d>1$, since there is no 1-dimensional noncommutative tori. However some of our results are actually compatible with the case when we allow $d=1$. 

\begin{thm}[Theorem \ref{main1}]
Given a dimension $d$ and an order $n>1$, there is an action of $\mathbb{Z}_n$ on a simple noncommutative $d$-tori $\mathcal{A}_{\Theta}$ if either $d-W(n)>1$ or $d-W(n)=0$. The crossed product $\mathcal{A}_{\Theta}\rtimes \mathbb{Z}_n$ is an AT algebra, and it is an AF algebra if and only if $d-W(n)=0$, and $n$ either admits a form of $2m$ where $m=3^j5^ip_l^{e_l}$, or admits a form of $2^k3^j5^i$ where $k\neq1$, $j\leq2$, $i\leq1$ and $e_l$ are all nonnegative integers, and $p_l>5$ is a prime number.

\end{thm}
Then we generalize the function $W$ for finite abelian groups, and realize actions with a similar argument. We obtain the following.
\begin{thm}[Theorem \ref{3.5}]\label{intro}
Given a dimension $d$ and an abelian finite subgroup $G\leq {\rm GL}_d(\mathbb{Z})$ with $d-W(G)>1$ or $d-W(G)=0$, there is an action of $G$ on  a simple noncommutative $d$-torus $\mathcal{A}_{\Theta}$. The crossed product $\mathcal{A}_{\Theta}\rtimes G$ is an AT algebra, and it is an AF algebra if and only if $d-W(G)=0$ and, if we write $G$ as $G\cong \prod_{j=1}^t\mathbb{Z}_{n_j}$ such that 
$$
W(G)=\sum_{l=1}^t W(\mathbb{Z}_{n_l}),
$$
each $n_l$ either admits a form of $2m$ where $m=3^j5^ip_r^{e_r}$, or admits a form of $2^k3^j5^i$ where $k\neq1$, $j\leq2$, $i\leq1$ and $e_r$ are all nonnegative integers, and $p_r>5$ is a prime number. 

\end{thm}
Especially we have Corollary \ref{cor} which states that the only possibility of a finite abelian group $F$ acting on a irrational rotation algebra is when $F=\mathbb{Z}_k$ where $k=2, 3, 4$ and $6$. In Section 3 we prove above theorems and discuss related results.

Then we find out that the function $W$ contains more information. We have the following.
\begin{thm}[Theorem \ref{no}]
Given a dimension $d$ and an order $n$, if $d-W(n)=1$, then there is no simple noncommutative torus on which $\mathbb{Z}_n$ acts in the above way.
\end{thm}
In Section 4 we give the proof of this theorem.

With enough knowledge of the function $W$ when considering cyclic group actions, we move to the topic of finitely generated abelian group acting on simple higher dimensional noncommutative tori. The action of torsion free part of the finitely generated group is defined to be adding dimensions concerning the crossed product. Similarly we study the condition by which the crossed product is an AF algebra.
\begin{thm}[Theorem \ref{main2}]
For a given dimension $d$ and a finitely generated abelian group $G\cong\Big( \prod_{i=1}^{s}\mathbb{Z}_{p_i^{e_i}} \Big)\times \mathbb{Z}^r$, if $d-W(n)=1$ and $r>0$, or $d-W(n)\neq1$ then there is a noncommutative $d$-torus $\mathcal{A}_{\Theta}$ admitting an action of $G$, denoted by $\alpha:G\rightarrow{\rm Aut}(\mathcal{A}_{\Theta})$. We require $\mathcal{A}_{\Theta}$ is pseudo-simple in the first case and simple in the second case. Then the crossed product $\mathcal{A}_{\Theta}\rtimes_{\alpha}G$ is a simple AT algebra, and it is an AF algebra if and only if $r=0$, $d-W(G)=0$ and G satisfies the last condition in Theorem \ref{intro}.
\end{thm}
A detailed discussion is in Section 5.


\section{Preliminaries}

For a second countable locally compact Hausdorff topological group $G$ with its modular function $\Delta_G:G\rightarrow (0, +\infty)$ and a Borel 2-cocycle $\omega\in Z(G,\mathbb{T})$, we define the associated reduced (resp. full) twisted group $C^*$-algebras, $C_r^*(G, \omega)$ (resp. $C^*(G, \omega)$) in the following way. Regard $L^1(G)$ as a vector space, and equip it with twisted convolution given by
$$
f\ast_{\omega}g(t)=\int_Gf(s)g(s^{-1}t)\omega(s,s^{-1}t)ds
$$
and the involution given by
$$
f^*(t)=\Delta_G(t^{-1})\overline{\omega(t,t^{-1})f(t^{-1})}.
$$
Then $L^1(G,\omega):=(L^1(G),\ast_{\omega},*)$ becomes a $\ast$-algebra which is named the {\it twisted convolution algebra}. As similar to the definition of group $C^*$-algebras, to consider nondegenerate representations of $L^1(G,\omega)$, we turn to a twisted analogue of unitary representations of $G$. More precisely, an {\it $\omega$-representation} of $G$ on a Hilbert space $\mathcal{H}$ is a Borel map $V:G\rightarrow \mathcal{U}(\mathcal{H})$ such that
$$
V(t)V(s)=\omega(t,s)V(ts)
$$ 
where $\mathcal{U}(\mathcal{H})$ stands for the unitary group of $\mathcal{H}$ endowed with the strong operator topology. For example we may define $L_{\omega}:G\rightarrow\mathcal{U}(L^2(G))$, the {\it regular $\omega$-representation} of $G$, by
$$
L_{\omega}(s)\xi(t)=\omega(s,s^{-1}t)\xi(s^{-1}t)
$$
for all $\xi\in L^2(G)$. For arbitrary $\omega$-representation $V:G\rightarrow\mathcal{U}(\mathcal{H})$, define $\pi:L^1(G,\omega)\rightarrow B(\mathcal{H})$ via
$$
\pi(f)=\int_Gf(s)V(s)ds
$$
for $f\in L^1(G,\omega)$. Then $\pi$ is a contractive $\ast$-homomorphism. As in the definition of a group $C^*$-algebra, with abuse of notations we may also denote $\pi$ by $V:L^1(G,\omega)\rightarrow B(\mathcal{H})$. Hence we can define the {\it reduced twisted group $C^*$-algebra} as
$$
C_r^*(G, \omega)=\overline{L_{\omega}(L^1(G,\omega))}.
$$
Moreover every nondegenerate representation of $L^1(G, \omega)$ is induced by an $\omega$-representation of $G$. We can see this by taking a norm one approximate identity in $L^1(G,\omega)$. 
Thus we obtain the universal representation $\pi_{u}$ of $L^1(G,\omega)$, and define the {\it full twisted group $C^*$-algebra} as
$$
C^*(G, \omega)=\overline{\pi_{u}(L^1(G,\omega))}.
$$
We remark that $C^*(G,\omega)$ has the universal property and the $\ast$-homomorphism $L_{\omega}:C^*(G, \omega)\rightarrow C_r^*(G, \omega)$ is surjective. The reduced and full twisted group $C^*$-algebras coincide if $G$ is amenable.

We realize higher dimensional noncommutative tori with twisted group $C^*$-algebras. For a given dimension $d$, take $G=\mathbb{Z}^d$. Throughout this paper we denote by $\mathcal{T}_d(\mathbb{R})$ the set of all $d\times d$ skew symmetric real matrices. Then for $\Theta=(\theta_{ij})\in \mathcal{T}_d(\mathbb{R})$ define the induced 2-cocycle $\omega_{\Theta}:\mathbb{Z}^d\times \mathbb{Z}^d\rightarrow \mathbb{T}$ by formula
$$
\omega_{\Theta}(x,y)=\exp(\pi i\langle \Theta x, y\rangle)
$$
for $x,y\in \mathbb{Z}^d$. Then $\mathcal{A}_{\Theta}:=C^*(\mathbb{Z}^d,\omega_{\Theta})$ is called a {\it noncommutative $d$-torus}. Since $\mathbb{Z}^d$ is discrete and amenable, if we take the standard basis of it, say $e_i $ for $i=1,\ldots, d$, and denote the regular $\omega_{\Theta}$-representation by $l_{\Theta}:\mathbb{Z}^d\rightarrow\mathcal{U}(\ell^2(\mathbb{Z}^d))$, then we have
$$
\mathcal{A}_{\Theta}:=C^*(\mathbb{Z}^d,\omega_{\Theta})=C^*\{l_{\Theta}(e_i)\mid i=1,\ldots, d \}
$$
where each $l_{\Theta}(e_i)$ is a unitary and they satisfy the commuting relations
$$
l_{\Theta}(e_i)l_{\Theta}(e_j)=\omega_{\Theta}(e_i,e_j)^2l_{\Theta}(e_j)l_{\Theta}(e_i)=\exp(2\pi i \theta_{ji})l_{\Theta}(e_j)l_{\Theta}(e_i).
$$
This is clearly a generalization of the rotation algebra to the $d$-dimensional case. We refer to the rotation algebra which is the case when $d=2$.

One may take the zero $d\times d$ matrix as $\Theta$ and then get $\mathcal{A}_{\Theta}=C(\mathbb{T}^d)$, however this is not so much in our interests since it is not simple. Thus we have the following.

\begin{definition}
The matrix $\Theta\in \mathcal{T}_d({\mathbb{R}})$ is said to be $nondegerate$ if whenever $x\in\mathbb{Z}^d$ satisfies $\exp(2\pi i\langle\Theta x,y\rangle)=1$ for all $y\in\mathbb{Z}^d$, we have $x=0$.
\end{definition}

\begin{thm}[{\cite[Theorem 1.9]{Phi}}]
Let $\Theta\in\mathcal{T}_d(\mathbb{R})$. Then the noncommutative $d$-torus $\mathcal{A}_{\Theta}$ is simple if and only if so is $\Theta$.
\end{thm}

\begin{prop}
For $\Theta_i\in\mathcal{T}_{d_i}({\mathbb{R}})$, $i=1,\ldots,m$, denote $\Theta:=\bigoplus_{i=1}^m\Theta_i\in \mathcal{T}_d({\mathbb{R}})$ where $d:=\sum_{i=1}^md_i$. Then $\Theta$ is nondegenerate if and only if each $\Theta_i$ is nondegenerate.
\end{prop}

\begin{proof}
By induction we only show the the case when $m=2$. If $\Theta_1$ and $\Theta_2$ are both nondegenerate, for $x=(x_1,x_2)\in\mathbb{Z}^d$ satisfying $\exp(2\pi i\langle\Theta x,y\rangle)=1$ for all $y\in\mathbb{Z}^d$, by taking $y_1\in\mathbb{Z}^{d_1}\times \{0\}$ and $y_2\in\{0\}\times \mathbb{Z}^{d_2}$ we obtain $\exp(2\pi i\langle\Theta_i x_i,y_i\rangle)=1$ for all $y_i\in\mathbb{Z}^{d_i}$, $i=1,2$. Then by nondegeneracy of $\Theta_1$ and $\Theta_2$, we obtain $x=(x_1, x_2)=0$, in other words $\Theta$ is nondegenerate.

Conversely if $\Theta$ is nondegenerate, then for $x_1\in \mathbb{Z}^{d_1}$ satisfying 
$$\exp(2\pi i\langle\Theta_1 x_1,y_1\rangle)=1$$
 for all $y_1\in\mathbb{Z}^{d_1}$, then we have for $x=(x_1, 0)\in \mathbb{Z}^d$,
 $$
 \exp(2\pi i\langle\Theta x,y\rangle)=\exp(2\pi i\langle\Theta_1 x_1,y_1\rangle)=1
 $$
 for all $y\in\mathbb{Z}^d$. Thus $x_1=0$, and similarly $x_2=0$. Hence $\Theta_1$ and $\Theta_2$ are both nondegenerate.
 \end{proof}

Then we discuss actions on noncommutatitive tori. We denote by $\mathcal{A}_{\Theta}$ a noncommutative $d$-torus as above, which is not necessarily simple. Through the regular representation $l_{\Theta}:\ell^1(\mathbb{Z}^d, \omega_{\Theta})\rightarrow B(\ell^2(\mathbb{Z}^d))$, we regard $\mathcal{A}_{\Theta}$ as a $C^*$-subalgebra in $B(\ell^2(\mathbb{Z}^d))$.
 
For a matrix $A\in {\rm GL}_d(\mathbb{Z})$, we have a unitary $u_A$ in $B(\ell^2(\mathbb{Z}^d))$ given by
$$
u_A\xi(x)=\xi(A^{-1}x)
$$
for $\xi\in \ell^2(\mathbb{Z}^d)$ and $x\in\mathbb{Z}^d$. Thus we have ${\rm Ad}\ u_A\in {\rm Aut}(B(\ell^2(\mathbb{Z}^d)))$. Then we have an action denoted by 
$$
{\rm Ad}\ u_{\bullet}:{\rm GL}_d(\mathbb{Z})\rightarrow {\rm Aut}(B(\ell^2(\mathbb{Z}^d))).
$$
Then for $\mathcal{A}_{\Theta}=C^*\{l_{\Theta}(e_i)\mid i=1,\ldots, d \}=\overline{\rm span}\{l_{\Theta}(x)\mid x\in \mathbb{Z}^d\}$, we have the fomula
$$
{\rm Ad}(u_A)(l_{\Theta}(x))\xi(y)=l_{(A^{-1})^t\Theta A^{-1}}(Ax)\xi(y)
$$
for $\xi\in\ell^2(\mathbb{Z}^d)$ and $y\in\mathbb{Z}^d$. Thus for $A\in {\rm GL}_d(\mathbb{Z})$ such that $\Theta=(A^{-1})^t\Theta A^{-1}$ we have ${\rm Ad}(u_A)\in{\rm Aut}(\mathcal{A}_{\Theta})$ by restriction. Then denote
$$
G_{\Theta}:=\{A\in{\rm GL}_d(\mathbb{Z})\mid \Theta=A^t\Theta A\},
$$
we obtain an action
$$
{\rm Ad}\ u_{\bullet}:G_{\Theta}\rightarrow {\rm Aut}(\mathcal{A}_{\Theta}).
$$
This is an action on $\mathcal{A}_{\Theta}$ which generalizes the action of ${\rm SL}_2(\mathbb{Z})$ on the rotation algebra $\mathcal{A}_{\theta}$ discussed in \cite{Echterhorff}. Particularly, when the dimension $d=2$, we have $G_{\Theta}={\rm SL}_2(\mathbb{Z})$.  We are especially interested in such actions of some subgroup $G$ of $G_{\Theta}$ on $\mathcal{A}_{\Theta}$.

Then we roughly state classification results for our cases, and refer to \cite{Echterhorff}, \cite{LEE},\cite{Phi} and \cite{Roh} for details. A simple $\mathcal{A}_{\Theta}$ is known to be with a unique tracial state and has tracial rank zero. For an action of a finite group on it, which has the tracial Rokhlin property, say $\alpha:G\rightarrow {\rm Aut}(\mathcal{A}_{\Theta})$, the resulting crossed product $\mathcal{A}_{\Theta}\rtimes_{\alpha}G$ becomes also simple, with a unique tracial state and with tracial rank zero. Particularly, for a finite subgroup $G\leq G_{\Theta}$, the action described in last paragraph of $G$ on simple $\mathcal{A}_{\Theta}$ is known to have the tracial Rokhlin property. Moreover the resulting crossed product $\mathcal{A}_{\Theta}\rtimes G$ satisfies the Universal Coefficient Theorem.
Thus the crossed products of such actions becomes classifiable. We also need the following theorem.

\begin{thm}[{\cite[Preposition 3.7]{Phi}}]
Let $\mathcal{A}$ be a simple infinite dimensional separable unital nuclear $C^*$-algebra with tracial rank zero and which satisfies the Universal Coefficient Theorem. Then $\mathcal{A}$ is a simple AH algebra with real rank zero and no dimension growth. If $K_*(\mathcal{A})$ is torsion free, $\mathcal{A}$ is an AT algebra. If, in addition, $K_1(\mathcal{A})=0$, then $\mathcal{A}$ is an AF algebra.
\end{thm}

Then we introduce results of J. A. Jeong and J. H. Lee in \cite{LEE}. We use their realization method as a building block to obtain more realizations and actions.

For a give dimension $d$ and an order $n$ with $d=\phi(n)$, where $\phi$ is the Euler's totient function. The $n$th {\it cyclotomic polynomial} $\Phi_{n}(x)$ is defined by 
$$
\Phi_n(x)=\prod_{\substack{0<k\leq n\\ \gcd(k,n)=1}} (x-\exp (2\pi i \frac{k}{n})).
$$
$\Phi_{n}(x)$ is known to be a monic polynomial of degree $d=\phi(n)$, and with integer coefficients. Thus set

$$
C_n=\left( \begin{array}{cccccc}
0&0&0&\cdots&0&-a_0\\
1&0&0&\cdots&0&-a_1\\
0&1&0&\cdots&0&-a_2\\
\vdots&\vdots&\ddots&\vdots&\vdots&\vdots\\
0&0&0&\ddots&0&-a_{d-2}\\
0&0&0&\cdots&1&-a_{d-1}\end{array} \right),
$$
which is the companion matrix of $\Phi_n(x)$. Then we have $C_n$ is in ${\rm GL}_d(\mathbb{Z})$ and of order $n$.

\begin{thm}[{\cite[Theorem 4.2]{LEE}}]
Let $n\geq 3$ and $d:=\phi(n)$. Then there exist simple $d$-dimensional tori on which the group $\mathbb{Z}_n=\langle C_n\rangle$ acts canonically.
\end{thm}
We remark in this theorem the \lq\lq canonical action\rq\rq\ is the action we discussed above and what we are interested in. To compute the $K$-theory for associated crossed product of such action, we need the following theorem.

\begin{thm}[{\cite[Theorem 0.1]{Luck}}] \label{Luck thm}
Let $n,d \in\mathbb{N}$. Consider the extension of groups 
$$
1\rightarrow \mathbb{Z}^d \rightarrow \mathbb{Z}^d\rtimes_{\alpha} \mathbb{Z}_n \rightarrow \mathbb{Z}_n \rightarrow 1
$$
such that conjugation action $\alpha$ of $\mathbb{Z}_n$ on $\mathbb{Z}^d$ is free outside of the origin $0\in \mathbb{Z}^d$. Then $K_0(C^*(\mathbb{Z}^d\rtimes_{\alpha} \mathbb{Z}_n))\cong \mathbb{Z}^{s_0}$ for some $s_0\in \mathbb{Z}$ and
$$
K_1(C^*(\mathbb{Z}^d\rtimes_{\alpha} \mathbb{Z}_n))\cong \mathbb{Z}^{s_1},
$$
where $s_1=\sum_{l\geq0}{\rm rk}_{\mathbb{Z}}((\bigwedge\nolimits^{2l+1}\mathbb{Z}^d)^{\mathbb{Z}_n})$. If $n$ is even, $s_1=0$. If $n>2$ is prime and $d=n-1$, then $s_1=\frac{2^{n-1}-(n-1)^2}{2n}$.

\end{thm}

The authors in \cite{LEE} apply this method and obtain that for certain dimension $d$ and order $n$, the resulting crossed product is not an AF algebra. The theorem is stated as the following.
\begin{thm}\cite[Theorem 3.6]{LEE}
Let $n\geq 7$ be an odd integer and $d:=\phi(n)$. Consider the extension of groups $1\rightarrow\mathbb{Z}^d\rightarrow\mathbb{Z}^d\rtimes_{\alpha}\mathbb{Z}_n\rightarrow\mathbb{Z}_n\rightarrow 1$ with $\mathbb{Z}_n=\langle C_n\rangle$. If $2d\geq n+5$, then
$$
K_1(C^*(\mathbb{Z}^d\rtimes_{\alpha}\mathbb{Z}_n))\neq0.
$$
\end{thm}

We denote by $\mathfrak{N}$ the smallest subcategory of the category of separable $C^*$-algebras,  which contains separable Type I algebras and is closed under taking ideals, quotients, extensions, inductive limits, stable isomorphisms and crossed products by $\mathbb{Z}$ and $\mathbb{R}$, then the following theorem holds.
\begin{thm}[{\cite[K\"unneth theorem]{Sch}}] \label{Kunneth}
Let $\mathcal{A}$ and $\mathcal{B}$ be $C^*$-algebras with $\mathcal{A}$ in $\mathfrak{N}$. Then there is a natural short exact sequence
$$
0\rightarrow K_*(\mathcal{A})\otimes K_*(\mathcal{B}) \xrightarrow{\alpha} K_*(\mathcal{A}\otimes\mathcal{B}) \xrightarrow{\beta} {\rm Tor}(K_*(\mathcal{A}),K_*(\mathcal{B})) \rightarrow 0.
$$
The sequence is $\mathbb{Z}_2$-graded with $\deg\alpha=1$, $\deg\beta=1$, where $K_p\otimes K_q$ and ${\rm Tor}(K_p\otimes K_q)$ are given degree $p+q$ for $p,q\in \mathbb{Z}_2$.

\end{thm}

\section{Realization of cyclic groups acting on noncomutative $d$-tori}

Given skew symmetric real matrices $\Theta_1$ and $\Theta_2$, suppose that $\Theta=\Theta_1\oplus \Theta_2$, then naturally we have 
$$
\mathcal{A}_\Theta\cong\mathcal{A}_{\Theta_1}\otimes \mathcal{A}_{\Theta_2}.
$$
Then for any finite group actions described in the Preliminaries of 
$
G_1\leq G_{\Theta_1}$ on $ \mathcal{A}_{\Theta_1}$, and 
$
G_2 \leq G_{\Theta_2} $ on $ \mathcal{A}_{\Theta_2},
$
we could naturally get a finite group action of
$
G_1 \times G_2
$ on
$
\mathcal{A}_{\Theta},
$
one can easily verify that $G=G_1\times G_2 \leq G_{\Theta}$ and the resulting action is the restriction of ${\rm Ad}\ u_{\bullet}:{\rm GL}_d(\mathbb{Z})\rightarrow {\rm Aut}(B(\ell^2(\mathbb{Z}^d))),$ where we denote by $d$ the degree of $\Theta$.
Moreover we have 
$$
\mathcal{A}_{\Theta} \rtimes G \cong (\mathcal{A}_{\Theta_1} \rtimes G_1) \otimes (\mathcal{A}_{\Theta_2} \rtimes G_2)
$$
by an isomorphism defined by
$$
(a\otimes b, (g,h))\mapsto (a,g) \otimes(b,h).
$$

For a given dimension $d$, consider ${\rm GL}_d(\mathbb{Z})$. By \cite[Theroem 2.7]{Kuz}, $\mathrm{GL}_d(\mathbb{Z})$ has an element of order $n=\prod_{i=1}^{t} p_i^{e_i}$ if and only $W(n)\leq d$, where $p_1<\cdots<p_t$ are all prime numbers and
$$
W(n):= \begin{cases}

\sum_{i=1}^t (p_i-1)p_i^{e_i-1}-1, &p_1^{e_1}=2;\\
\sum_{i=1}^t (p_i-1)p_i^{e_i-1},& {\rm otherwise,}
\end{cases}
$$
thus we are able to obtain finitely many candidates for $n$, which is the possible order of a realizable cyclic group in ${\rm GL}_d(\mathbb{Z})$. For those $n$'s with $\phi(n)=d$, where $\phi$ is the Euler's totient function, by \cite[Theorem 4.2]{LEE}  we may find a noncommutative $d$-torus $\mathcal{A}_{\Theta}$ which is simple, namely figure out the form of the nondegenerate $\Theta$, and realize an action of $\mathbb{Z}_n$ on $\mathcal{A}_{\Theta}$. Moreover we could compute $K_*(\mathcal{A}_{\Theta}\rtimes \mathbb{Z}_n )$ and then know whether $\mathcal{A}_{\Theta}\rtimes \mathbb{Z}_n$ is AF or not. However this way does not work for those $n$'s with $\phi(n)\neq d$, paticularly when $d$ is odd. We manage a way of realizing cyclic groups for more candidates $n$ to find actions in the following.

By Kuzmanovich's proof in \cite{Kuz}, suppose $n=\prod_{i=1}^{t} p_i^{e_i}$ with $p_1<\cdots<p_t$ such that $W(n)\leq d$. If $p_1^{e_1}\neq2$, by applying the companion matrix we could construct a $d_i\times d_i$ matrix $A_i$ of order $p_i^{e_i}$ where $d_i=(p_i-1)p_i^{e_i-1}$. We note that $\phi(p_i^{e_i})=(p_i-1)p_i^{e_i-1}$. Then set $B:=\bigoplus_{i=1}^t A_i$, which is a $W(n)\times W(n)$ matrix with $W(n)\leq d$.

Thus by \cite[Theorem 4.2]{LEE}, for each $i$, we could find nondegenerate $\Theta_i$ in $\mathcal{T}_{d_i}(\mathbb{R})$ such that $\langle A_i \rangle=\mathbb{Z}_{p_i^{e_i}}$ acts on $\mathcal{A}_{\Theta_i}$.
If $d-W(n)>1$, set $A:=B\oplus I_{d-W(n)}$, which is a $d\times d$ matrix of order $n$.  It is possible to find a nondegenerate $\Theta_0$. Then set 
$$
\Theta:=\Big(\bigoplus_{i=1}^t\Theta_i\Big)\oplus \Theta_0,
$$
which is nondegenerate, thus we have an action of
$
\langle A \rangle =\mathbb{Z}_{n}$ on $ \mathcal{A}_{\Theta}.
$
Moreover by our former discussion one may check that
$$
\mathcal{A}_{\Theta}\rtimes \mathbb{Z}_n\cong \Big(\bigotimes_{i=1}^t(\mathcal{A}_{\Theta_i}\rtimes\mathbb{Z}_{p_i^{e_i}}) \Big)\otimes \mathcal{A}_{\Theta_0}.
$$

By the method above we may realize more finite cyclic group actions on higher dimensional noncommutative tori. To compute their $K$-groups, we apply the K\"unneth theorem in \cite{Sch}, stated as Theorem \ref{Kunneth}.

To verify the conditions for our problem, we have the following confirmation.  Firstly by \cite[Proposition 3.4]{LEE} each $\mathcal{A}_{\Theta_i}\rtimes\mathbb{Z}_{p_i^{e_i}}$ is an AT algebra, and so is $\mathcal{A}_{\Theta_0}$. Since $C(\mathbb{T})\otimes F$ is separable and of Type I, where $F$ stands for a finite dimensional $C^*$-algebra, then every $\mathcal{A}_{\Theta_i}\rtimes\mathbb{Z}_{p_i^{e_i}}$ and $\mathcal{A}_{\Theta_0}$ are in $\mathfrak{N}$. Secondly by \cite[Remark 3.2]{LEE} and \cite[Theorem 0.1]{Luck} we know that for each $i$,
$$
K_*(\mathcal{A}_{\Theta_i}\rtimes\mathbb{Z}_{p_i^{e_i}})\cong K_*(C^*(\mathbb{Z}^{d_i}\rtimes\mathbb{Z}_{p_i^{e_i}}, \Tilde{\omega}_{\Theta_i})\cong K_*(C^*(\mathbb{Z}^{d_i}\rtimes\mathbb{Z}_{p_i^{e_i}})), 
$$
which is torsion free.

Then applying K\"unneth theorem, we obtain
$$
K_*(\mathcal{A}_{\Theta}\rtimes \mathbb{Z}_n)=\Big(\bigotimes_{i=1}^t K_*(\mathcal{A}_{\Theta_i}\rtimes \mathbb{Z}_{p_i^{e_i}})\Big)\otimes K_*(\mathcal{A}_{\Theta_0}).
$$
By \cite[Propersition 3.1]{LEE} $\mathcal{A}_{\Theta}\rtimes \mathbb{Z}_n$ satisfies the Universal Coefficient Theorem. As we have already stated before, $\mathcal{A}_{\Theta}\rtimes \mathbb{Z}_n$ is then classifiable. We obtain the $K_*(\mathcal{A}_{\Theta}\rtimes \mathbb{Z}_n)$ is torsion free, thus under this realization the crossed product $\mathcal{A}_{\Theta}\rtimes \mathbb{Z}_n$ is an AT algebra. However, We note that $d-W(n)>1$ and
$$
K_0(\mathcal{A}_{\Theta_0})\cong K_1(\mathcal{A}_{\Theta_0})\cong \mathbb{Z}^{2^{d-W(n)}-1}.
$$
Thus in this case $K_1(\mathcal{A}_{\Theta}\rtimes \mathbb{Z}_n)\neq 0$, and $\mathcal{A}_{\Theta}\rtimes \mathbb{Z}_n$ is not an AF-algebra.

If $d-W(n)=0$, the situation is essentially the same to our discussion before. Instead we set
$
A:=B=\bigoplus_{i=1}^t A_i
$, which is a $d\times d$ matrix of order $n$. Then set
$$
\Theta:=\Big(\bigoplus_{i=1}^t\Theta_i\Big),
$$
which is nondegenerate, thus we have an action of $\langle A\rangle=\mathbb{Z}_n$ on $\mathcal{A}_{\Theta}.$ Then we obtain
$$
\mathcal{A}_{\Theta}\rtimes \mathbb{Z}_n\cong \bigotimes_{i=1}^t(\mathcal{A}_{\Theta_i}\rtimes\mathbb{Z}_{p_i^{e_i}}) 
$$
and again by the K\"unneth theorem,
$$
K_*(\mathcal{A}_{\Theta}\rtimes \mathbb{Z}_n)=\bigotimes_{i=1}^t K_*(\mathcal{A}_{\Theta_i}\rtimes \mathbb{Z}_{p_i^{e_i}}).
$$
Thus $K_1(\mathcal{A}_{\Theta}\rtimes \mathbb{Z}_n)= 0$ if and only if $K_1(\mathcal{A}_{\Theta_i}\rtimes \mathbb{Z}_{p_i^{e_i}})=0$ for each $i$.

Then our attention is driven to each $\mathcal{A}_{\Theta_i}\rtimes \mathbb{Z}_{p_i^{e_i}}$. Since in the  case when $p_i$ is an odd prime, by \cite[Theorem 3.6]{LEE}, we have $K_1(\mathcal{A}_{\Theta_i}\rtimes \mathbb{Z}_{p_i^{e_i}})\neq 0$ if $p_i^{e_i-1}(p_i-2)\geq 5$. Thus the only candidates of $K_1(\mathcal{A}_{\Theta_i}\rtimes \mathbb{Z}_{p_i^{e_i}})= 0$ are $p_i^{e_i}=3, 3^2, 5$ or $2^k$ where $k>1$. To compute $K_1$-groups, we apply \cite[Theorem 0.1]{Luck}, stated as Theorem \ref{Luck thm}.

Hence the corresponding $K_1$-groups vanish when $p_i^{e_i}=3,5$ and $2^k$ where $k>1$. In the following we only check the case when $p_i^{e_i}=3^2$. The theorem above enables us to only compute $s_1$ where
$$
s_1={\rm rk}_{\mathbb{Z}}((\bigwedge\nolimits ^1\mathbb{Z}^6)^{\mathbb{Z}_9})+{\rm rk}_{\mathbb{Z}}((\bigwedge\nolimits ^3\mathbb{Z}^6)^{\mathbb{Z}_9})+{\rm rk}_{\mathbb{Z}}((\bigwedge\nolimits ^5\mathbb{Z}^6)^{\mathbb{Z}_9}).
$$
We obtain ${\rm rk}_{\mathbb{Z}}((\bigwedge\nolimits ^1\mathbb{Z}^6)^{\mathbb{Z}_9})=0$ for the action is free outside the origin $0 \in \mathbb{Z}^6$. According to our realization we have $\mathbb{Z}_9$ is the cyclic group generalized by $A$ in ${\rm GL}_6(\mathbb{Z})$, where
$$
A=\left( \begin{array}{cccccc}
0&0&0&0&0&-a_0\\
1&0&0&0&0&-a_1\\
0&1&0&0&0&-a_2\\
0&0&1&0&0&-a_3\\
0&0&0&1&0&-a_4\\
0&0&0&0&1&-a_5\end{array} \right)
$$ and
$$\Phi_9(x)=\prod_{\substack{0<k<9\\ \gcd(k,9)=1}} (x-\zeta^k)=\sum_{i=0}^{5}a_i x^i, \zeta=\exp(2\pi i \frac{1}{9}).$$
Write 
$$
A=Q^{-1}{\rm diag}(\zeta, \zeta^2, \zeta^4, \zeta^5, \zeta^7, \zeta^8)Q,
$$
where $Q\in {\rm GL}_6(\mathbb{C})$.
Notice that $x$ a fixed point of $A$ acting on $\bigwedge^3 \mathbb{Z}$ if and only if $Qx$ is a fixed point of ${\rm diag}(\zeta, \zeta^2, \zeta^4, \zeta^5, \zeta^7, \zeta^8)$ acting on $\bigwedge^3 Q\mathbb{Z}$. Then for $1\leq i<j<k\leq 6$, if 
$$
{\rm diag}(\zeta, \zeta^2, \zeta^4, \zeta^5, \zeta^7, \zeta^8) e_i\wedge e_j \wedge e_k =e_i\wedge e_j \wedge e_k ,
$$
it is the same to stating that there is a way of pick three number in ${1,2,4,5,7,8}$ such that the sum of them can be divided by $9$, which is impossible. Thus we obtain
$$
{\rm rk}_{\mathbb{Z}}((\bigwedge\nolimits ^3\mathbb{Z}^6)^{\mathbb{Z}_9})=0.
$$
Similarly we also obtain
$$
{\rm rk}_{\mathbb{Z}}((\bigwedge\nolimits ^5\mathbb{Z}^6)^{\mathbb{Z}_9})=0.
$$
Hence when $p_i^{e_i}=3^2$, we obtain a simple noncommutative 6-torus  $\mathcal{A}_{\Theta}$ on which $\mathbb{Z}_9$ acts in the way we are interested in, and 
$$
K_1(\mathcal{A}_{\Theta}\rtimes \mathbb{Z}_9)= 0,
$$
in other words $\mathcal{A}_{\Theta}\rtimes \mathbb{Z}_9$ is an AF algebra.

We leave the case when $d-W(n)=1$ to next section. To complete the proof, we discuss the case when given the dimension $d$ and the order $n$ such that $W(n)\leq d$, and with $p_1^{e_1}=2$ in the prime factorization of $n=\prod_{i=1}^t p_i^{e_i}$. By proof in \cite{Kuz} notice that $W(\frac{n}{2})=W(n)\leq d$. Then for $\frac{n}{2}$, by applying the companion matrix we could construct a $d_i\times d_i$ matrix $A_i$ of order $p_i^{e_i}$ where $d_i=(p_i-1)p_i^{e_i-1}$ for $i=2,\ldots,t$. Then set $B:=\bigoplus_{i=2}^t A_i$, which is a $W(n)\times W(n)$ matrix with $W(n)\leq d$, and of order $\frac{n}{2}$.

Again we could find nondegenerate $\Theta_i$ in $\mathcal{T}_{d_i}(\mathbb{R})$ such that $\langle A_i\rangle$ acts on $\mathcal{A}_{\Theta_i}$ in the way we described above for $i=2,\ldots,t$. If $d-W(n)=d-W(\frac{n}{2})> 1$, set $A:=B\oplus I_{d-W(n)}$, which is a $d\times d$ matrix of order $\frac{n}{2}$. It is possible to find a nondegenerate $\Theta_{0}$ since $d-W(n)>1$. Then set
$$
\Theta:=\Big(\bigoplus_{i=1}^t\Theta_i\Big)\oplus \Theta_0,
$$
which is nondegenerate. Since
$$
(-A)^t\Theta(-A)=A^t\Theta A=\Theta,
$$
we have an action of $\langle -A\rangle=\mathbb{Z}_n$ on $\mathcal{A}_{\Theta}$. However, in this case we can only write $\mathbb{Z}_n$ as 
$$
\mathbb{Z}_n\cong \Big( \prod_{\substack{2\leq i\leq t\\ i\neq j}}\mathbb{Z}_{p_i^{e_i}} \Big)\times \mathbb{Z}_{2p_j^{e_j}}
$$
where $2\leq j \leq t$. Thus instead of $A$, we consider $A_{(j)}$ as a generator of $\mathbb{Z}_n$ where $A_{(j)}$ is defined by
$$
A_{(j)}=\left( \begin{array}{ccccccc}
A_1&0&\cdots&0&0&0&0\\
0&A_2&\cdots&0&0&0&0\\
\vdots&\vdots&\ddots&\vdots&\vdots&\vdots&\vdots\\
0&0&\cdots&-A_j&0&0&0\\
\vdots&\vdots&\vdots&\vdots&\ddots&\vdots&\vdots\\
0&0&\cdots&0&0&A_t&0\\
0&0&\cdots&0&0&0&I_{d-W(n)}\end{array} \right)
$$
where $1\leq j\leq t$. Then we consider actions of $\mathbb{Z}_n=\langle A_{(j)}\rangle$ on 
$\mathcal{A}_{\Theta}$.

Similarly we have 
$$
\mathcal{A}_{\Theta}\rtimes \mathbb{Z}_n\cong \Big(\bigotimes_{\substack{2\leq i\leq t\\ i\neq j}}(\mathcal{A}_{\Theta_i}\rtimes\mathbb{Z}_{p_i^{e_i}}) \Big)\otimes (\mathcal{A}_{\Theta_j}\rtimes \mathbb{Z}_{2p_j^{e_j}})\otimes\mathcal{A}_{\Theta_0},
$$
where $2\leq j \leq t$.
By the K\"unneth therorem $K_*(\mathcal{A}_{\Theta}\rtimes \mathbb{Z}_n)$ is the tensor of $K_*$ of each factor.
Again since the part of $K_*(\mathcal{A}_{\Theta_0})$, 
$
K_1(\mathcal{A}_{\Theta}\rtimes \mathbb{Z}_n)\neq0,
$
which is equivalently to say that $\mathcal{A}_{\Theta}\rtimes \mathbb{Z}_n$ is not an AF algebra.

It is similar to above when $d-W(n)=d-W(\frac{n}{2})=0$. However we remark that 
$$
K_*(\mathcal{A}_{\Theta}\rtimes \mathbb{Z}_n)\cong \Big(\bigotimes_{\substack{2\leq i\leq t\\ i\neq j}}K_*(\mathcal{A}_{\Theta_i}\rtimes\mathbb{Z}_{p_i^{e_i}}) \Big)\otimes K_*(\mathcal{A}_{\Theta_j}\rtimes \mathbb{Z}_{2p_j^{e_j}})
$$
where $2\leq j \leq t$.
Thus by \cite[Theorem 0.1]{Luck}, we have $K_1(\mathcal{A}_{\Theta_j}\rtimes \mathbb{Z}_{2p_j^{e_i}})=0$ since $2p_j^{e_j}$ is even. Hence in this case, $\mathcal{A}_{\Theta}\rtimes \mathbb{Z}_n$ is an AF-algebra exactly when $\frac{n}{2}$ admits a form of $3^j5^ip_l^{e_l}$ where $j$, $i$ and $e_l$ are all nonnegetive integers, $j\leq 2$, $i\leq 1$ and $p_l>5$ is a prime number. As a summarize we have the following theorem.
\begin{thm}\label{main1}
Given a dimension $d$ and an order $n>1$, there is an action of $\mathbb{Z}_n$ on a simple noncommutative $d$-tori $\mathcal{A}_{\Theta}$ if either $d-W(n)>1$ or $d-W(n)=0$. The crossed product $\mathcal{A}_{\Theta}\rtimes \mathbb{Z}_n$ is an AT algebra, and it is an AF algebra if and only if $d-W(n)=0$, and $n$ either admits a form of $2m$ where $m=3^j5^ip_l^{e_l}$, or admits a form of $2^k3^j5^i$ where $k\neq1$, $j\leq2$, $i\leq1$ and $e_l$ are all nonnegative integers, and $p_l>5$ is a prime number. 

\end{thm}
\begin{cor}\label{iff2}
Given a dimension $d$ which is even and order $n$, there is an action of $\mathbb{Z}_n$ on simple noncommutative $d$-tori if and only if $W(n)\leq d$.
\end{cor}
\begin{proof}
Since $W(n)$ is always even when $n\geq 3$.
\end{proof}

Note that we have the following example which is already mentioned in \cite{Echterhorff}.
\begin{example}
Consider an action of $\mathbb{Z}_n$ on a simple $2$-dimensional noncommutative torus , say $\mathcal{A}_{\Theta}$. Suppose that the action is of the above type. By Corollary \ref{iff2}, it is equivalent to say that $d-W(n)=0$ or  $d-W(n)=2$ where $d=2$. Then by knowledge of $W(n)$, we obtain $n=2,3,4$ or $6$. Moreover by Theorem \ref{main1}, each crossed product $\mathcal{A}_{\Theta}\rtimes \mathbb{Z}_n$ is an AF algebra.
\end{example}

Then generally for a finite abelian group $G$, we may generalize our method of construct such actions of $G$ on noncommutative tori. 
By the structure theorem for finitely generated abelian groups we write $G$ as
$$
G\cong  \prod_{i=1}^{s}\mathbb{Z}_{p_i^{e_i}}
$$
where $p_1\leq\cdots\leq p_s$ are primes and $e_i\leq e_{i+1}$ whenever $p_i=p_{i+1}$. However we cannot simply define $W(G):=\sum_{i=1}^sW(p_i^{e_i})$. Since if so, suppose the case when $G=\mathbb{Z}_2\times\mathbb{Z}_2$, in which we have $W(G)=0$, but we can only realize it in ${\rm GL}_4(\mathbb{Z})$ as $\langle A, B\rangle$ where
$$
A=\left( \begin{array}{cccc}
-1&0&0&0\\
0&-1&0&0\\
0&0&1&0\\
0&0&0&1\end{array} \right)\ {\rm and }\ 
B=\left( \begin{array}{cccc}
1&0&0&0\\
0&1&0&0\\
0&0&-1&0\\
0&0&0&-1\end{array} \right).
$$
This is because if for a finite abelian group $G\cong \prod_{i=1}^{s}\mathbb{Z}_{p_i^{e_i}}$, if $\mathbb{Z}_2$ is a normal group of it, i.e. $p_1^{e_1}=2$, and there is no other $p_i^{e_i}$ such that $p_i^{e_i}\neq 2$, then we cannot write $\mathbb{Z}_2$ into $\mathbb{Z}_{2p_i^{e_i}}$ as a normal subgroup, where $\mathbb{Z}_{2p_i^{e_i}}$ is a normal subgroup of $G$. Thus we cannot realize $\mathbb{Z}_2$ as the signature part of the matrix which need not a cost of a dimension, and since the determinant of realized matrix is $1$, the degree should be even. Thus instead we have to cost two of dimensions to set a diagonal part as example above. Hence in this case we should modify our test funtion to taking value $2$ of $\mathbb{Z}_2$.

\begin{definition}
For a cyclic group $\mathbb{Z}_n$, define the function $W$ as follows. 
$$
W(\mathbb{Z}_n):= \begin{cases}

W(n), & n\neq 2;\\
2,& n=2.
\end{cases}
$$
\end{definition}

\begin{definition}
For a finitely generated abelian group $G$, by the structure theorem for finitely generated abelian groups, we have 
$$
G\cong \Big( \prod_{i=1}^{s}\mathbb{Z}_{p_i^{e_i}} \Big)\times \mathbb{Z}^r,
$$ 
where $p_1\leq\cdots\leq p_s$ are primes and $e_i\leq e_{i+1}$ whenever $p_i=p_{i+1}$. Let $G_{\rm tor}$ denote the torsion part  $G_{\rm tor}$.

Define the function $W$ of $G$ as
$$
W(G):=\min \Big\{\sum_{l=1}^t W(\mathbb{Z}_{n_l})\mid G_{\rm tor}\cong \prod_{j=1}^{l}\mathbb{Z}_{n_j}\Big\}.
$$

\end{definition}
\begin{thm} \label{3.5}
For a given dimension $d$ and an abelian finite subgroup $G\leq {\rm GL}_d(\mathbb{Z})$ with $d-W(G)>1$ or $d-W(G)=0$, there is an action of $G$ on  a simple noncommutative $d$-torus $\mathcal{A}_{\Theta}$. The crossed product $\mathcal{A}_{\Theta}\rtimes G$ is an AT algebra, and it is an AF algebra if and only if $d-W(G)=0$ and, if we write $G$ as $G\cong \prod_{j=1}^t\mathbb{Z}_{n_j}$ such that 
$$
W(G)=\sum_{l=1}^t W(\mathbb{Z}_{n_l}),
$$
each $n_l$ either admits a form of $2m$ where $m=3^j5^ip_r^{e_r}$, or admits a form of $2^k3^j5^i$ where $k\neq1$, $j\leq2$, $i\leq1$ and $e_r$ are all nonnegative integers, and $p_r>5$ is a prime number. 
\end{thm}

\begin{proof}
By definition we can write $G$ as $G\cong \prod_{j=1}^t\mathbb{Z}_{n_l}$ such that 
$$
W(G)=\sum_{l=1}^t W(\mathbb{Z}_{n_l}).
$$
Thus we have $d-W(\mathbb{Z}_{n_l})>1$ or $d-W(\mathbb{Z}_{n_l})=0$ for all $l$. Then for each $j$ by Theorem \ref{main1} we obtain a simple noncommutative torus $\mathcal{A}_{\Theta_l}$ on which $\mathbb{Z}_{n_l}$ acts in the above way. We emphasis that if $n_l=2$, we obtain an irrational rotation algebra. Then by the tensor product argument similar in section 3 we obtain a simple noncommutative $d$-torus $\mathcal{A}_{\Theta}$ on which $G$ acts in the above way.

The rest of the proof is similar to the one in section 3 and from Theorem \ref{main1}.
\end{proof}

\begin{cor} \label{cor}
The only actions of finite abelian group on irrational rotation algebras of the above type are by $\mathbb{Z}_k$ where $k=2,3,4$ and $6$.
\end{cor}
\begin{proof}
Note that in the current case $d=2$. suppose $G$ is a finite abelian group acting on the irrational rotation algebra $\mathcal{A}_{\theta}$ in the above way. It is a conflict if $W(G)>2$ then by our statement above. Thus $W(G)=2$. Then by the definition of the function $W$ we have $G=\mathbb{Z}_k$ where $k=2, 3, 4$ and $6$.
\end{proof}

\section{Case $d-W(n)=1$ and actions on odd-dimensional noncommutative tori}
As a remaining problem of the last section, we need to study the case $d-W(n)=1$ for a given dimension $d$ and order $n=\prod_{i=1}^{t} p_i^{e_i}$ with primes $p_1<\cdots<p_t$. We start our discussion again from Kuzmanovich's result which shows that for any $A\in {\rm GL}_d(\mathbb{Z})$ of order $n$, there is $Q\in {\rm GL}_d(\mathbb{Q})\subseteq {\rm GL}_d(\mathbb{R})$ such that
$$
\Lambda:=QAQ^{-1}=\left( \begin{array}{ccccc}
A_1&0&\cdots&0&0\\
0&A_2&\cdots&0&0\\
\vdots&\vdots&\ddots&\vdots&\vdots\\
0&0&\cdots&A_t&0\\
0&0&\cdots&0&1\end{array} \right)
$$
where $A_i$ is the companion matrix for $p_i^{e_i}$.

In general for a subgroup $G\leq {\rm GL}_d(\mathbb{Z})$ and an element $Q\in{\rm GL}_d(\mathbb{R})$ with $QGQ^{-1}\leq {\rm GL}_d(\mathbb{Z})$, it is routine to check that $G\leq G_{\Theta_1}$ if and only if $
QGQ^{-1}\leq G_{\Theta_2}
$
where $\Theta_2:=(Q^{-1})^t\Theta_1 Q^{-1}$.

Consider a category $\mathfrak{G}$ defined as following: an object is a pair $(G, \omega)$ where $G$ is a discete group and $\omega$ is a 2-cocycle in $Z(G,\mathbb{T})$. For objects $(G_1, \omega_1)$ and $(G_2, \omega_2)$,
$$
{\rm Hom}_{\mathfrak{G}}((G_1, \omega_1),(G_2, \omega_2)):=\{\varphi: G_1\rightarrow G_2\mid \omega_2\circ \varphi=\omega_1\}.
$$
If we denote by $\mathfrak{C}$ the category of $C^*$-algebras and define
$$
C^*(\cdot, \cdot)(G,\omega):=C^*(G,\omega),
$$
$$
C^*(\cdot, \cdot)\varphi: \ell^1(G_2,\omega_2)\rightarrow \ell^1(G_1,\omega_1), f\mapsto f\circ\varphi.
$$

By the universal property of a full twisted $C^*$-algebra, one may check that $C^*(\cdot, \cdot)$ is a functor from $\mathfrak{G}$ to $\mathfrak{C}$.

Thus for arbitrary $\Theta_1\in \mathcal{T}_d(\mathbb{R})$ and $Q\in {\rm GL}_d(\mathbb{R})$, we have $\Theta_2:=Q^t\Theta_1Q\in \mathcal{T}_d(\mathbb{R})$. By taking different generators we can obtain an isomorphism 
 $$Q^{-1}: \mathbb{Z}^d\rightarrow\mathbb{Z}^d,$$  
moreover
$$
\omega_{\Theta_2}\circ Q^{-1}(x,y)=\exp (\pi i\langle \Theta_2 Q^{-1}x,Q^{-1}y\rangle)=\exp (\pi i\langle \Theta_1 x,y\rangle)=\omega_{\Theta_1}(x,y).
$$
Thus we obtain
$$
\mathcal{A}_{\Theta_1}=C^*(\mathbb{Z}^d, \omega_{\Theta_1})\cong C^*(\mathbb{Z}^d, \omega_{\Theta_2})=\mathcal{A}_{\Theta_2}.
$$

Since $\mathcal{T}_{d, A}(\mathbb{R})=\mathcal{T}_{d, Q^{-1}\Lambda Q}(\mathbb{R})=Q^t\mathcal{T}_{d, \Lambda}(\mathbb{R})Q$, and we have the following preposition in general.

\begin{prop}
For $\Lambda\in{\rm GL}_d(Z)$ where
$$
\Lambda=\left( \begin{array}{cccc}
A_1&0&\cdots&0\\
0&A_2&\cdots&0\\
\vdots&\vdots&\ddots&\vdots\\
0&0&\cdots&A_t\end{array} \right)
$$
where each $A_i$ is the companion matrix for $p_i^{e_i}$ of degree $d_i$, and for $\Theta\in\mathcal{T}_{d, \Lambda}(\mathbb{R})$, then 
$$
\Theta=\left( \begin{array}{cccc}
\Theta_{11}&\Theta_{12}&\cdots&\Theta_{1t}\\
-\Theta_{12}^t&\Theta_{22}&\cdots&\Theta_{2t}\\
\vdots&\vdots&\ddots&\vdots\\
-\Theta_{1t}^t&-\Theta_{2t}^t&\cdots&\Theta_{tt}\end{array} \right)
$$
where for $i\leq j$, $\Theta_{ij}$ is a $d_i\times d_j$ matrix , and
$
\Theta_{ij}=(\theta_{kl})
$
such that $\theta_{kl}=\theta_{(k+1)(l+1)}$ for all $k,l$.
\end{prop}
\begin{proof}
Write $\Theta=(\Theta_{ij})$ then by $\Lambda^t\Theta\Lambda=\Theta$ we obtain
$$
A_i^t\Theta_{ij}A_j=\Theta_{ij}.
$$
Then by a similar proof to the one in \cite[Lemma 4.3]{LEE}, we maintain the form of $\Theta_{ij}$.
\end{proof}
\begin{prop}
$\Theta\in\mathcal{T}_{d, \Lambda}(\mathbb{R})$ has the following form,
$$
\Theta=\left( \begin{array}{cc}
\Theta'&0\\
0&0\end{array} \right)
$$
for some $\Theta' \in \mathcal{T}_{d-1}(\mathbb{R})$.
\end{prop}
\begin{proof}
We write $\Theta_{d_i}:=(\theta_{dk})$ where $\sum_{l=1}^{i-1}d_l<k\leq\sum_{l=1}^{i}d_l$. Then for
$
\Lambda^t\Theta\Lambda=\Theta,
$
we have 
$$
\Theta_{d_i}A_i=\Theta_{d_i}.
$$
Thus it suffices to show when $i=1$. Since $A_1$ is the companion matrix for $p_1^{e_1}$, we have 
$$
(\theta_{d,1}, \theta_{d,2},\ldots, \theta_{d,d_1-1}, \theta_{d,d_1})=(\theta_{d,2}, \theta_{d,3},\ldots, \theta_{d,d_1}, -\sum_{l=0}^{d_1-1}a_l\theta_{d,l+1})
$$
where
$$
\sum_{l=0}^{d_1-1}a_lx^l=\Phi_{p_1^{e_1}}(x)=\sum_{k=0}^{p_i-1}x^{ip_i^{e_i-1}}.
$$
Hence we have $\theta:=\theta_{d,1}=\cdots=\theta_{d,d_1}$ and 
$$
-\Big(\sum_{l=0}^{d_1-1}a_l\Big)\theta=\theta,
$$
thus $\theta=0$ and then $\Theta_{d_i}=0$.
\end{proof}

Thus for any $d$-dimensional noncommutative torus $\mathcal{A}_{\Theta}$ on which $\mathbb{Z}_n=\langle A\rangle$ acts, we have 
$
\Theta=Q^t\Theta''Q
$
where
$$
\Theta'':=\left( \begin{array}{cc}
\Theta'&0\\
0&0\end{array} \right).
$$ 
Then we have  $\mathcal{A}_{\Theta}\cong\mathcal{A}_{\Theta''}\cong\mathcal{A}_{\Theta'}\otimes C(\mathbb{T})$, which is clearly not simple.

Hence we have the following theorem.
\begin{thm} \label{no}
For a given dimension $d$ and order $n$, if $d-W(n)=1$, then there is no simple noncommutative tori on which $\mathbb{Z}_n$ acts in the above way.
\end{thm}
\begin{cor}\label{iff}
For a given dimension $d$ and order $n$, there is a action of $\mathbb{Z}_n$ on a simple noncommutative $d$-torus $\mathcal{A}_{\Theta}$ if and only if either $d-W(n)>1$ or $d-W(n)=0$. 
\end{cor}

\begin{cor}
For a given dimension $d>3$ which is odd and an order $n$, then there is an action of $\mathbb{Z}_n$ on simple $d$-dimensional noncommutative tori if and only if there is an action of $\mathbb{Z}_n$ of the above type on $d-3$-dimensional noncommutative tori.
\end{cor}
\begin{proof}
Notice that $d-W(n)$ is an odd integer then by Corollary \ref{iff} and Corollary \ref{iff2} we draw the conclusion.
\end{proof}

We mention that Theorem \ref{no} is a generalization of \cite[Theorem 5.2]{LEE} by the following example.
\begin{example}
Consider an action of $\mathbb{Z}_n$ on a simple $3$-dimensional noncommutative torus. Then by Theorem \ref{no} and by the fact that $W(n)$ is always even, it is equivalent to say that $d-W(n)=3$. Thus we obtain $n=2$, in other words, the only action by a nontrivial finite cyclic group on a simple $3$-dimensional torus of the above type is the flip action by $\mathbb{Z}_2$, as stated in \cite[Theorem 5.2]{LEE}.
\end{example}

We rewrite a fact mentioned in \cite[Lemma 1.5]{Phi} as following. In general for a $d$-dimensional noncommutative torus $\mathcal{A}_{\Theta}$ where $\Theta\in \mathcal{T}_d(\mathbb{R})$, write
$$
\Theta=\left( \begin{array}{cc}
\Theta'&(\theta_{id})\\
-(\theta_{id})^t&0\end{array} \right)
$$ 
and the standard generator of $\mathcal{A}_{\Theta}$ as $u_i$ for $i=1,\ldots, d$. Then 
$$
\mathcal{A}_{\Theta'}=C^*\{u_i\mid i=1,\ldots, d-1\},
$$
moreover we can define an action $\alpha: \mathbb{Z}\rightarrow{\rm Aut}(\mathcal{A}_{\Theta'})$ by $\alpha(1)(u_i)=\exp(2\pi i\theta_{id})u_i$ which is homotopic to the identity and 
$$
\mathcal{A}_{\Theta}\cong\mathcal{A}_{\Theta'}\rtimes_{\alpha}\mathbb{Z}.
$$

Return to our current case for the given dimension $d$ and the order $n$ with $d-W(n)=1$. We can construct a noncommutative $d$-torus $\mathcal{A}_{\Theta}$ where $\Theta\in\mathcal{T}_d(\mathbb{R})$ and
$$
\Theta=\left( \begin{array}{cc}
\Theta'&0\\
0&0\end{array} \right),
$$
and $\mathcal{A}_{\Theta}$ admits an action of $\mathbb{Z}_n$. Although $\mathcal{A}_{\Theta}\cong\mathcal{A}_{\Theta'}\otimes C(\mathbb{T})$ is never simple, or equivalently $\Theta$ is never nondegenerate, by our construction $\Theta'$ is nondegenerate, and the simple noncommutative $d-1$-tori $\mathcal{A}_{\Theta'}$ admits an action of $\mathbb{Z}_n$ which is denoted by $\alpha'$. Hence write $u_i$ for $i=1,\ldots,d$ the standard generators of $\mathcal{A}_{\Theta}$, we could define an action
$$
\alpha:\mathbb{Z}\times\mathbb{Z}_n\rightarrow{\rm Aut}(\mathcal{A}_{\Theta})
$$
by the following,
$$
\begin{cases}
\alpha(1, I)(u_i)=u_i, & i=1,\ldots, d-1;\\
\alpha(1, I)(u_i)=\exp(2\pi i \theta)u_i, & i=d;\\
\alpha(0, A)(u_i)=\alpha'(u_i), & i=1,\ldots, d-1;\\
\alpha(0,A)(u_i)=u_i, & i=d
\end{cases}
$$
where $\theta$ is an irrational number. It is actually the product action of two commuting actions. Then we have 
$$
\mathcal{A}_{\Theta}\rtimes_{\alpha}(\mathbb{Z}\times\mathbb{Z}_n)\cong(\mathcal{A}_{\Theta}\rtimes_{\alpha}\mathbb{Z})\rtimes_{\tilde{\alpha}}\mathbb{Z}_n\cong(\mathcal{A}_{\Theta'}\otimes\mathcal{A}_{\theta})\rtimes_{\alpha'}\mathbb{Z}_n.
$$
Since 
$$
\Theta''=\Theta'\oplus\left( \begin{array}{cc}
0&\theta\\
-\theta&0\end{array} \right)
$$
is simple, and action $\alpha'$ of $\mathbb{Z}_n$ on $\mathcal{A}_{\Theta''}\cong\mathcal{A}_{\Theta'}\otimes\mathcal{A}_{\theta}$ is of the above type, the the crossed product $\mathcal{A}_{\Theta}\rtimes_{\alpha}(\mathbb{Z}\times\mathbb{Z}_n)$ is covered by a discussion in Section $3$. We comment that action $\alpha$ restricted to $\mathbb{Z}$ is adding dimension to noncommutative torus $\mathcal{A}_{\Theta}$ and making it a simple one on which there is an action of $\mathbb{Z}_n$ of above type, in our context this should be the action of $\mathbb{Z}$ on a noncommutative torus in which we are interested.

\section{Finitely generated abelian group actions}

Motived by the last part of the former section, we have the following discussion.

For a given dimension $d$ and a finitely generated abelian group $G$, by the structure theorem we have 
$$
G\cong \Big( \prod_{i=1}^{s}\mathbb{Z}_{p_i^{e_i}} \Big)\times \mathbb{Z}^r,
$$
if $d-W(G)=0$ or $d-W(G)>0$, there is an action of $\prod_{i=1}^{s}\mathbb{Z}_{p_i^{e_i}} $ on a simple noncommutative $d$-torus $\mathcal{A}_{\Theta}$ by Theorem \ref{3.5}, say $\alpha'$. If $r>0$ we may define an action $\alpha:G\rightarrow{\rm Aut}(\mathcal{A}_{\Theta})$ with a nondegenerate $\Theta_0\in\mathcal{T}_r(\mathbb{R})$ such that
$$
\mathcal{A}_{\Theta}\rtimes_{\alpha}G\cong(\mathcal{A}_{\Theta}\otimes\mathcal{A}_{\Theta_0})\rtimes_{\alpha'}\prod_{i=1}^{s}\mathbb{Z}_{p_i^{e_i}}.
$$
Note that there is also a similar action when $r=0$ stated in Theorem \ref{3.5}.

If $d-W(n)=1$, there is a noncommutative $d$-torus $\mathcal{A}_{\Theta}$ of form $\mathcal{A}_{\Theta'}\otimes C(\mathbb{T})$, or equivalently, with
$$
\Theta=\left( \begin{array}{cc}
\Theta'&0\\
0&0\end{array} \right)
$$
where $\Theta'\in\mathcal{T}_{d-1}(\mathbb{R})$ is nondegenerate. Both $\mathcal{A}_{\Theta}$ and $\mathcal{A}_{\Theta'}$ admit an action of $\mathbb{Z}_n$. Denote the latter one by $\alpha'$. If additionally $r>0$, we can obtain an action $\alpha:G\rightarrow{\rm Aut}(\mathcal{A}_{\Theta})$. In such case if $r>1$ we obtain $\alpha$ with an irrational number $\theta$ and a nondegenerate $\Theta_0\in\mathcal{T}_{r-1}(\mathbb{R})$ such that 
$$
\mathcal{A}_{\Theta}\rtimes_{\alpha}G\cong(\mathcal{A}_{\Theta'}\otimes\mathcal{A}_{\theta}\otimes\mathcal{A}_{\Theta_0})\rtimes_{\alpha'}\prod_{i=1}^{s}\mathbb{Z}_{p_i^{e_i}}.
$$
Similarly if $r=1$ we obtain $\alpha$ with only an irrational number $\theta$ such that
$$
\mathcal{A}_{\Theta}\rtimes_{\alpha}G\cong(\mathcal{A}_{\Theta'}\otimes\mathcal{A}_{\theta})\rtimes_{\alpha'}\prod_{i=1}^{s}\mathbb{Z}_{p_i^{e_i}}.
$$

\begin{thm}\label{main2}
For a given dimension $d$ and a finitely generated abelian group $G\cong\Big( \prod_{i=1}^{s}\mathbb{Z}_{p_i^{e_i}} \Big)\times \mathbb{Z}^r$, If $d-W(n)=1$ and $r>0$, or $d-W(n)\neq1$ then there is a noncommutative $d$-torus $\mathcal{A}_{\Theta}$ admitting an action of $G$, denoted by $\alpha:G\rightarrow{\rm Aut}(\mathcal{A}_{\Theta})$. We require $\mathcal{A}_{\Theta}$ is pseudo-simple in the first case and simple in the second case. Then the crossed product $\mathcal{A}_{\Theta}\rtimes_{\alpha}G$ is a simple AT algebra, and it is an AF algebra if and only if $r=0$, $d-W(G)=0$ and G satisfies the last condition in Theorem \ref{3.5}.
\end{thm}
\begin{proof}
The existence is by our discussion in this section. Then by Theorem \ref{3.5} we draw the conclusion.
\end{proof}

\section*{Acknowledgment} 
The author would like to express his deep gratitude to his supervisor, Professor Yasuyuki Kawahigashi, for helpful suggestions, warm encouragement and support during two years of his master course. The author would like to thank Dr. Chen Jiang, Dr. Yuki Arano, Dr. Yosuke Kubota, Dr. Shuhei Masumoto, Dr. Takuya Takeishi and Dr. Lu Xu for many useful discussions. The author is grateful to the University of Tokyo for Special Scholarship for International Students (Todai Fellowship).


\begin{thebibliography}{99}

\bibitem{Echterhorff} S. Echterhorff, W. L\"uck, N. C. Phillips, S. Walters, \emph{The structure of crossed products of irrational rotation algebras by finite subgroups of ${\rm SL}_2(\mathbb{Z})$}, J. Reine Angew. Math. {\bf 639} (2010), 173--221.

\bibitem{LEE} J. A. Jeong, J. H. Lee, \emph{Finite groups acting on higher dimensonal noncommutative tori}, arXiv:1402.1826v2.



\bibitem{Kuz} J. Kuzmanovich, A. Pavlichenkov, \emph{Finite groups of matrices whose entires are intergers}, Amer. Math. Monthly {\bf 109} (2002), 173--186.

\bibitem{Luck} M. Langer, W. L\"uck, \emph{Topological $K$-theroy of the group $C^*$-algebra of a semi-direct product $\mathbb{Z}^n\rtimes \mathbb{Z}/m$ for a free conjugation action}, J. Topol. Anal. {\bf4} (2012), 121--172.

\bibitem{Packer} J. A. Packer, I. Raeburn, \emph{Twisted crossed products of $C^*$-algebras}, Math. Proc. Cambridge Philos. Soc. {\bf 106} (1989), 293--311.

\bibitem{Phi} N. C. Phillips, \emph{Every higher noncommutative Torus is an AT-algebra}, arXiv:math/0609783.

\bibitem{Roh} N. C. Phillips, \emph{The tracial Rokhlin property for actions of finite groups on $C^*$-algebras}, Amer. J. Math. {\bf 133} (2011), 581--636.






\bibitem{Sch} C. Schochet, \emph{Topological methods for $C^*$-algebras II: Geometric resolutions and the K\"unneth formula}, Pacific J. Math. {\bf 98} (1982), 443--458.




\end{thebibliography}
\end{document}